\documentclass{birkau}
\usepackage{amsmath}
\usepackage{amssymb}
\usepackage{enumerate}
\usepackage[mathscr]{eucal}
\usepackage{xcolor,pgf,tikz}

\newtheorem{theo}{Theorem}[section]
\newtheorem{coro}[theo]{Corollary}

\newtheorem{lemma}[theo]{Lemma}

\theoremstyle{definition}

\newtheorem*{pr}{Problem}

\def\Con{\operatorname{Con}}

\def\ISP{\operatorname{ISP}}
\def\I{\operatorname{I}}
\def\S{\operatorname{S}}
\def\P{\operatorname{P}}

\def\M{\operatorname{M}}

\def\rng{\operatorname{rng}}

\newcommand{\twiddle}[1]{\smash{\underset{\smash{\raise.2ex\hbox{$\sim$}}}
              {\mathbf{#1}}}\vphantom{#1}}
\def\MT{\twiddle M}

\def\A{\mathbb A}
\def\B{\mathbb B}
\def\M{\mathbb M}
\def\L{\mathbb L}
\def\X{\mathbb X}

\def \c {^{\circ}}
\def \cc{^{\circ\circ}}
\usetikzlibrary{calc,positioning,shapes,arrows,arrows.meta}
\tikzset{%
 shaded/.style={draw, shape=circle, fill=black!35, inner sep=1.4pt},
 unshaded/.style={draw, shape=circle, fill=white, inner sep=1.4pt},
 order/.style={thin},
 label/.style={shape=rectangle, inner sep=6pt},
 auto,
 map/.style={->, densely dashed, shorten >=5pt, shorten <=5pt, >=stealth', looseness=1.1},
 curvy/.style={->, densely dashed, shorten >=5pt, shorten <=12pt, >=stealth', looseness=1.1},
 operationgj/.style={->, densely dashed, shorten >=5pt, shorten <=20pt, >=stealth', looseness=1.1},
 relationlejk/.style={->, shorten >=5pt, shorten <=5pt, >=stealth'},
 }

\begin{document}

%%%%%%%%%%%%%%%%%%%%%%%%%%%%%%%%%%%%%%%%%%%%%%%%%%%%%%%%%%%%%%%%%%%%%%
%% FRONT MATTER
%%%%%%%%%%%%%%%%%%%%%%%%%%%%%%%%%%%%%%%%%%%%%%%%%%%%%%%%%%%%%%%%%%%%%%

\title[Perfect
extensions of de Morgan algebras]{Perfect
extensions of de Morgan algebras}

%% First author: in the order \author, \address, \urladdr, \email
\author[M. Haviar]{Miroslav Haviar}
\address{Department of Mathematics\\Faculty of Natural Sciences,\\
M. Bel University\\Tajovsk\'eho
40, 974~01 Bansk\'a Bystrica\\Slovakia\\
and\\
Department of Mathematics\\
and Applied Mathematics\\
University of Johannesburg\\PO Box 524, Auckland Park, 2006\\South~Africa
}
\email{miroslav.haviar@umb.sk}

%% Second author: in the order \author, \address, \urladdr, \email
\author[M. Plo\v s\v cica]{Miroslav Plo\v s\v cica}
\address{Institute of Mathematics, Faculty of Natural Sciences\\
\v Saf\'arik's University\\
Jesenn\'a 5\\
041 54 Ko\v sice, Slovakia}
\email{miroslav.ploscica@upjs.sk}

\thanks{The second author acknowledges support from Slovak grant VEGA 1/0097/18.}

\dedicatory{Dedicated to the memory of Dr. Milan Demko (1963 -- 2021)}



\begin{abstract}
An algebra $\A$ is called a \textit{perfect extension} of its
subalgebra~$\B$ if every congruence of $\B$ has a unique extension to
$\A$. This terminology was used by Blyth and Varlet [1994].
In the case of lattices, this concept was described by Gr\"atzer and Wehrung [1999] 
by saying that $\A$ is a \textit{congruence-preserving extension} of $\B$.

Not many investigations of this concept have been carried out so
far. The present authors in another recent study faced the question of when 
a de Morgan algebra $\M$ is perfect extension of its Boolean
subalgebra $B(\M)$, the so-called \textit{skeleton} of $\M$.
In this note a~full solution to this interesting problem is given.
The theory of natural dualities in the sense of Davey and Werner [1983]
and Clark and Davey [1998], as well as Boolean product representations, 
are used as the main tools to obtain the solution.
\end{abstract}

\maketitle

%%%%%%%%%%%%%%%%%%%%%%%%%%%%%%%%%%%%%%%%%%%%%%%%%%%%%%%%%%%%%%%%%%%%%%
%% MAIN MATTER
%%%%%%%%%%%%%%%%%%%%%%%%%%%%%%%%%%%%%%%%%%%%%%%%%%%%%%%%%%%%%%%%%%%%%%

\section{Introduction}\label{Intro}
Blyth and Varlet introduced MS-algebras as algebras abstracting de Morgan and
Stone algebras in~\cite{BV1} (see also \cite{BV3}). In \cite{BHP19}  we
presented, for a special class of MS-algebras, an analogue of
Gr\"atzer's problem \cite[Problem 57]{G71}, which was formulated for
distributive p-algebras. %such algebras abstracting de Morgan and Stone algebras.
Compared to the solution for the distributive p-algebras by Katri\v
n\'ak in \cite{K76}, we provided in \cite{BHP19} a much simpler 
though not very descriptive solution, and a short and elegant proof
of it. That is why we then  considered in \cite{BHP19} an analogue 
of the Gr\"atzer problem in the special case when the MS-algebra
$\L$ we investigated was a \textit{perfect extension} of its
largest Stone subalgebra. This case was subsequently reduced to
the property that a de Morgan subalgebra $\M$ of $\L$ was a perfect
extension of its Boolean skeleton $B(\M)$.

The description of general de Morgan algebras $\M$ which are perfect
extensions of their Boolean subalgebras $B(\M)$ turns out to be an
interesting problem in its own right, and the answer has not, until now, been known. 
In this paper we give a satisfactory solution to this problem.
We show that a de Morgan
algebra $\M$ is a perfect extension of its Boolean skeleton $B(\M)$ if
and only if $\M$ is a Boolean product of copies of
the four-element subdirectly irreducible de Morgan
algebra and its three- and two-element subalgebras, and that this is
equivalent to a simple and nice condition on the natural dual space
of $\M$.

\section{Preliminaries}\label{Prel}

By a \textit{de Morgan-Stone algebra} (or \textit{MS-algebra}) we mean 
an algebra $\L= (L; \vee, \wedge, ^{\c},0,
1)$ of type $(2,2,1,0,0)$ where $(L; \vee, \wedge, 0, 1)$ is a
bounded distributive lattice and $^{\c}$ is a unary operation such
that for all $x,y \in L$,
 $x\leq x^{\cc},\ (x\wedge y)^{\c}=x^{\c}\vee y^{\c},\ 1^{\c}=0$.

An important subvariety of MS-algebras is that of \textit{de Morgan
algebras}, which satisfy the additional identity $x=x^{\cc}$. For a de
Morgan algebra $\L$ its \textit{skeleton} is the Boolean subalgebra
$B(\L) =
\{x\in L \mid x \vee x^{\c} = 1\} $. % of complemented elements of $L$.

We also mention other distinguished subvarieties of
MS-algebras are those of \textit{Kleene algebras}, which are de
Morgan algebras satisfying the identity $(x \wedge x^{\c}) \vee y
\vee y^{\c} = y \vee y^{\c}$, the subvariety of \textit{Stone
algebras} characterized by the identity $x \wedge x^{\c} = 0$, and
the subvariety of \textit{Boolean algebras} determined by the
identity $x \vee x^{\c} = 1$.

It is also appropriate to recall here some other concepts.  An
algebra $\A$ satisfies the \textit{Congruence Extension Property}
(briefly ({CEP})) if for every subalgebra $\B$ of $\A$ and every
congruence $\theta$ of $\B$, $\theta$ extends to a congruence of $\A$.
An algebra $\A$ is said to be a \textit{perfect extension} of its
subalgebra $\B$, if every congruence of $\B$ has a unique extension to
$\A$ (see \cite[page 30]{BV3}). We notice that in the literature such
an $\A$ is also called a \textit{congruence-preserving extension} of
$\B$; we refer to a paper by Gr\"atzer and Wehrung \cite{GW} where
this concept is used in case of lattices. Of course, if $\A$ is a
perfect extension of $\B$, then the congruence lattices $\Con(\A)$ of
$\A$ and $\Con(\B)$ of $\B$ are isomorphic.

The classes of distributive lattices and of MS-algebras are known to
satisfy the (CEP) \cite{G71}, \cite{BV3}. However, very little seems
to be known when  in these or other classes of algebras, an algebra
$\A$ is a perfect extension of its subalgebra $\B$. In \cite{GW},
Gr\"atzer and Wehrung proved that every lattice with more than one
element has a proper congruence-preserving extension. In
\cite{BHP19} we showed that a so-called principal MS-algebra $\L$ is
a perfect extension of its largest Stone subalgebra $\L_{S}= \{x\in
L \mid x^{\c} \vee  x^{\cc} = 1\}$ if and only if the de Morgan
subalgebra $\L^{\cc}= \{x\in L \mid x = x^{\cc}\}$ of $\L$ is a
perfect extension of its Boolean skeleton $B(\L)$.

This leads to a natural problem, which will be formulated and 
addressed in the next section.

\section{Problem of perfect extensions for de Morgan algebras}\label{PE}

In this section we focus on the following problem: 

\begin{pr}\label{prob}
\textit{Describe de Morgan algebras $\M$ which are perfect extensions of
their Boolean skeletons $B(\M)$.}
\end{pr}

An answer to this problem has not been known. The main result of
our paper is a satisfactory solution to this problem. We characterize
de Morgan algebras that are perfect extensions of their Boolean
skeletons both algebraically and via dual spaces. For the latter we
find it convenient to employ the theory of natural dualities as
described in~\cite{DW83} and \cite{CD}.

It is well known that the variety of de Morgan algebras is equal to
$\ISP(\M_1)$, where $\M_1$ is the four-element subdirectly
irreducible de Morgan algebra, and $\I$, $\S$ and $\P$ denote the
well-known operators of \textit{isomorphic copies},
\textit{subalgebras} and \textit{products}, respectively (for
example, see \cite[4.3.15]{CD}). In the natural duality for the
variety $\ISP(\M_1)$ based on $\M_1$, the dual spaces $\X
=(X;f,\preccurlyeq,\tau)$ are the usual Priestley spaces
$(X;\preccurlyeq,\tau)$ endowed with an order-reversing
homeomorphism $f$ of order two \cite{CF77}. Every de Morgan algebra
is isomorphic to the set of all morphisms from its dual space $\X$
into the \textit{alter ego} of the algebra $\M_1$, which is the 
structure $\MT_1 = (\{0,a,b,1\};f,\preccurlyeq,\tau)$ (see
\cite[Theorem 4.3.16]{CD}). The alter ego is the ordered set
$(\{0,a,b,1\};\preccurlyeq)$ with $a$ and $b$ being the top and the
bottom elements, respectively, and the two atoms $0,1$ (see Figure~\ref{fig:M1}). 
The homeomorphism $f$ is defined by
$f(0)=0$, $f(1)=1$, $f(a)=b$, $f(b)=a$, and $\tau$ is
the discrete topology. The set of morphisms
inherits the de Morgan algebra structure from the power algebra
$\M_1^X$.

%%%%%%%%%%%%%%%%%%%%%%%%%%
\begin{figure}[ht]
\centering
\begin{tikzpicture}[scale=0.625]
  % \M_1
  \begin{scope}
    \node[anchor=north] at (-1,0) {$\M_1$};
    % Elements
    \node[unshaded] (00) at (0,0) {};
    \node[unshaded] (10) at (2,2) {};
    \node[unshaded] (01) at (-2,2) {};
    \node[unshaded] (11) at (0,4) {};
    % Order
    \draw[order] (00) -- (10) -- (11);
    \draw[order] (00) -- (01) -- (11);
    % Labels
    \node[label,anchor=west] at (00) {$0 = 1^\circ$};
    \node[label,anchor=west] at (10) {$b=b^\circ$};
    \node[label,anchor=east] at (01) {$a=a^\circ$};
    \node[label,anchor=east] at (11) {$1 = 0^\circ$};
  \end{scope}
  % \MT_1
  \begin{scope}[xshift=8cm]
    \node[anchor=north] at (-1,0) {$\MT_1$};
    % Elements
    \node[unshaded] (00) at (0,0) {};
    \node[unshaded] (10) at (2,2) {};
    \node[unshaded] (01) at (-2,2) {};
    \node[unshaded] (11) at (0,4) {};
    % Order
    \draw[order] (00) -- (10) -- (11);
    \draw[order] (00) -- (01) -- (11);
    % Labels
    \node[label,anchor=west] at (00) {$b$};
    \node[label,anchor=west] at (10) {$1$};
    \node[label,anchor=east] at (01) {$0$};
    \node[label,anchor=east] at (11) {$a$};
  \end{scope}
\end{tikzpicture}
\caption{The de Morgan algebra ${\mathbb M}_1$ and its alter
ego.}\label{fig:M1}
\end{figure}
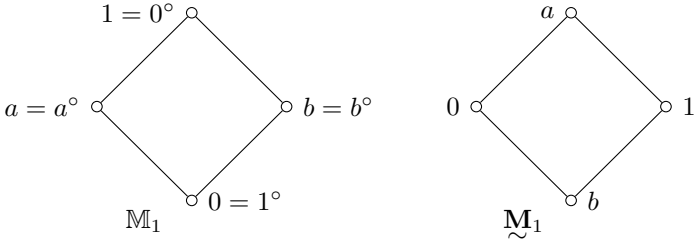
%%%%%%%%%%%%%%%%%%%%%%%%%%%%%%%%%%%

In our characterization we also use the concept of a Boolean product
(see \cite{UA}). A Boolean product of an indexed family $(\A_y)_{y\in
Y}$ of algebras, for a non-empty set $Y$, is a subdirect product $\A\le\Pi_{y\in Y}\A_y$, where
the set $Y$ can be endowed with a Boolean space topology so that
\begin{itemize}
\item[\rm(i)] the set $E(a,b)=\{y\in Y\mid a(y)=b(y)\}$ is clopen for every $a,b\in A$;
\item[\rm(ii)] if $a,b\in A$ and $K\subseteq Y$ is clopen, then $a\restriction K\cup b\restriction (Y\setminus K)\in A$.
\end{itemize}
The condition (ii) is sometimes called the patchwork property.

To shorten the proof of the main theorem, we first prove the following technical lemma.
\begin{lemma} Let $\X=(X;f,\preccurlyeq,\tau)$ be a Priestley space endowed with an order-reversing homeomorphism
$f$ of order two. Choose a
subset $Y\subseteq X$ satisfying the following conditions:
\begin{itemize}
\item[(a)] every $x\in X$ with $f(x)=x$ belongs to $Y$;
\item[(b)] if $x\ne f(x)$, then exactly one of these two elements belongs to $Y$.
\end{itemize}
Then the system
$$\sigma=\{Z\subseteq Y\mid Z\cup f(Z)\text{ is open in } \X\}$$
defines a Boolean topology on $Y$.\label{topol}
\end{lemma}
\begin{proof}
We leave it for the reader
to check that $\sigma$ is indeed a topology on $Y$. (But
notice that $Y$ is not a topological subspace of $\X$.) Now we check
the equality
$$(Y\setminus Z)\cup f(Y\setminus Z)=X\setminus (Z\cup f(Z))$$
for every $Z\subseteq Y$.
First, let $t\in Y\setminus Z$. We need to show that $t\in X\setminus (Z\cup f(Z))$,
that is  $t\notin f(Z)$. For a contradiction, let $t=f(u)$, $u\in Z$.
Then $u, f(u)\in Y$, which is only possible if $u=f(u)=t$, which contradicts $t\notin Z$.
Next, let $t\in f(Y\setminus Z)$. Then $t=f(u)$ for some $u\in Y\setminus Z$.
The injectivity of $f$ implies $t\notin f(Z)$. If $t\in Z$, then $u,f(u)\in Y$, hence $u=f(u)=t\in Z$,
a contradiction. So, $t\in X\setminus(Z\cup f(Z))$.

Conversely, let $t\in X\setminus (Z\cup f(Z))$. Then $t\notin f(Z)$ and $t=f(f(t))$ imply $f(t)\notin Z$.
If $t\in Y$, then clearly $t\in Y\setminus Z$.
If $t\notin Y$, then $f(t)\in Y\setminus Z$, so $t=f(f(t))\in f(Y\setminus Z)$.

As a consequence we obtain that  $Z\subseteq Y$ is closed (clopen) if and only if
$Z\cup f(Z)$ is closed (clopen) in $\X$. Let us check that the
space $Y$ is Boolean. Let $x,y\in Y$, $x\ne y$. Then also $x\ne
f(y)$, as either $f(y)=y$ or $f(y)\notin Y$. There exists a clopen
set $U\subseteq X$ with $x\in U$ and $y,f(y)\notin U$. Since $f$ is
a topological homeomorphism, the sets $f(U)$ and $U\cup f(U)$ are
also clopen. Let
$V := Y\cap(U\cup f(U))$.
Then $V\cup f(V)=U\cup f(U)$.
Indeed, from $V\subseteq U\cup f(U)$ %implies
it follows $f(V)\subseteq f(U\cup f(U))=f(U)\cup U$.
Conversely, if $u\in U\cup f(U)$, then either $u\in Y$ and hence $u\in V$, or $f(u)\in Y$ and hence
$f(u)\in V$ and $u\in f(V)$.

Hence, $V$ is clopen in $Y$ and $x\in V$. Since $f(y)\notin U$, we immediately have $y\notin U\cup f(U)$, so
$y\notin V$. Thus, every two points of $Y$ can be
separated by a clopen set. To check the compactness of $Y$, let
$\{Z_i\mid i\in I\}$ be an open cover of $Y$. Then $\{Z_i\cup
f(Z_i)\mid i\in I\}$ is an open cover of $\X$. Since this space is
compact, there exists a finite subset $I_0\subseteq I$ such that
$\{Z_i\cup f(Z_i)\mid i\in I_0\}$ covers $\X$. Then it is easy to
check that $\{Z_i\mid i\in I_0\}$ covers $Y$.
\end{proof}

\begin{theo} Let $\M$ be a de Morgan algebra. The following %conditions
are equivalent:
\begin{itemize}
\item[\rm(1)] $\M$ is a perfect extension of its Boolean skeleton $B(\M)$.
\item[\rm(2)] $\M$ is a Boolean product of copies of $\M_1$ and its subalgebras $\{0,a,1\}$ and
$\{0,1\}$.
\item[\rm(3)] In the dual space $\X$ of $\M$, $x\preccurlyeq y$ implies $x=y$ or $x=f(y)$ for all $x,y\in X$.
\end{itemize}
\end{theo}

\begin{proof} (2)$\Longrightarrow$(1) Let $\M\le\Pi_{y\in Y} \A_y$ be a Boolean product, where each $\A_y$
is equal to $\M_1$ or $\{0,a,1\}$ or $\{0,1\}$. The elements of $\M$
have the form $u=(u(y))_{y\in Y}$. The skeleton $B(\M)$ consists
exactly of those elements of $\M$ whose every coordinate is $0$ or
$1$. For every clopen set $N\subseteq Y$ we define
$$s_N(y)=\left\{\begin{array}{ll} 0\quad{\rm if}\ y\in N,\\
                               1\quad\text{if}\ y\notin N.
\end{array}\right.$$
Since constant $0$ and constant $1$ are elements of $\M$, the patchwork property implies
that $s_N\in M$ for every $N$.
Clearly, $s_N\in B(\M)$, so let $t_N$ denote its complement. We need to
prove that every congruence of $\M$ is determined by its restriction
to $B(\M)$.

For every $u,v\in M$, the set $E(u,v)=\{y\in Y\mid u(y)=v(y)\}$ is
clopen. Considering the cases $v=0$ and $v=1$ we obtain that the
sets $\{y\in Y\mid u(y)=0\}$, $\{y\in Y\mid u(y)=1\}$, and
consequently, $\{y\in Y\mid u(y)\in\{a,b\}\}$, are clopen.

Let $\theta\in\Con (\M)$. We claim that
$$(u,v)\in\theta\quad\text{if and only if}\quad (0,s_{E(u\wedge v, u\vee v)})\in\theta.$$
It suffices to show it for $u\le v$. Consider the sets
\begin{align}
J := &\{y\in Y\mid u(y)=0,\ v(y)=1\},\notag\\
K := &\{y\in Y\mid u(y)=0,\ v(y)\in\{a,b\}\},\notag\\
L := &\{y\in Y\mid u(y)\in\{a,b\},\ v(y)=1\}.\notag
\end{align}
These sets are clopen and $J\cup K\cup L$ is the complement of
$E(u,v)$, so it follows that $t_{J\cup K\cup L}=s_{E(u,v)}$.

Assume that $(u,v)\in\theta$. Then $(0,t_J)=(u\wedge t_J,v\wedge
t_J)\in\theta$. Further, $u\wedge t_K=0$, so $(0,v\wedge
t_K)\in\theta$, which implies that $(1,v^\circ\vee t_K^\circ)\in
\theta$. Taking the meet with $t_K$ we obtain that $(t_K, (v^\circ\vee
t_K^\circ)\wedge t_K)=(t_K,v^\circ\wedge t_K)\in\theta$. Clearly,
$v^\circ\wedge t_K=v\wedge t_K$, whence we get
$(0,t_K)\in\theta$. Similarly, $(u\wedge t_L,v\wedge t_L)=(u\wedge
t_L,t_L)\in\theta$, which implies $(u^\circ\vee
t_L^\circ,t_L^\circ)\in\theta$. Taking the meet with $t_L$ and using
$u^\circ\wedge t_L=u\wedge t_L$ we obtain that $(0,u\wedge
t_L)\in\theta$, hence also $(0,t_L)\in \theta$. Consequently,
$(0,s_{E(u,v)})=(0,t_J\vee t_K\vee t_L)\in\theta$.

Conversely, let $(0,s_{E(u,v)})\in\theta$. It is easy to see that
$(0\vee u)\wedge v=u$ and $(s_{E(u,v)}\vee u)\wedge v=v$. Hence
$(u,v)\in\theta$.

(1)$\Longrightarrow$(3)  Assume that $\M$  is equal to the set of
all morphisms from the dual space $\X$ into $\MT_1$ and
that (3) is not satisfied. Hence, we have $x\preccurlyeq y$ such
that $x\ne y$ and $x\ne f(y)$. Then $\{x,f(x)\}$ and $\{y,f(y)\}$
are distinct closed substructures of $\X$, so they determine
distinct congruences on $\M$, namely
\begin{align}
\alpha &=\{(\varphi,\psi)\in M^2\mid
\varphi\restriction{\{x,f(x)\}}=\psi\restriction{\{x,f(x)\}}\},\notag\\
\beta &=\{(\varphi,\psi)\in M^2\mid
\varphi\restriction{\{y,f(y)\}}=\psi\restriction{\{y,f(y)\}}\}.\notag
\end{align}

 We claim
that $\alpha$ and $\beta$ coincide on $B(\M)$. Take morphisms $\varphi,\psi\in
B(\M)$, hence we have for their ranges $\rng(\varphi), \rng(\psi)\subseteq\{0,1\}$. Since the
maps $\varphi,\psi$ preserve $\preccurlyeq$, we have $\varphi(x)\preccurlyeq\varphi(y)$
and $\psi(x)\preccurlyeq\psi(y)$. For elements in $\{0,1\}$ this means $\varphi(x)=\varphi(y)$ and
$\psi(x)=\psi(y)$. The morphisms commute with $f$, so
$\varphi(f(x))=f(\varphi(x))=f(\varphi(y))=\varphi(f(y))$ and similarly  $\psi(f(x))=\psi(f(y))$. Hence,
$(\varphi,\psi)\in\alpha$ if and only if $(\varphi,\psi)\in\beta$. This proves that $\M$ is not a
perfect extension of $B(\M)$.

(3)$\Longrightarrow$(2) As in the previous part, assume that $\M$ is
equal to the set of all morphisms from the dual space $\X$
into $\MT_1$. Let (3) be satisfied. Choose a
subset $Y\subseteq X$ satisfying the conditions (a) and (b) from Lemma \ref{topol}.

The set $M' :=\{\varphi\restriction Y\mid \varphi\in M\}$ clearly forms a
subalgebra of $\M_1^Y$. The assignment
$\varphi\mapsto\varphi\restriction Y$ is a surjective homomorphism
$\M\to \M'$. It is also injective: if $\varphi\ne\psi$, then
$\varphi(x)\ne\psi(x)$ for some $x\in X$. If $x\notin Y$, then
$f(x)\in Y$ and $\varphi(f(x))=f(\varphi(x))\ne
f(\psi(x))=\psi(f(x))$, as $f$ is a bijection. Hence, $\M$ is
isomorphic to $\M'$ and we will prove that $\M'$ is a Boolean product.

By Lemma \ref{topol}, the space $Y$ is Boolean. For every $y\in Y$, the set
$A_y=\{\varphi(y)\mid \varphi\in M\}$ forms a subalgebra $\A_y$ of $\M_1$, so it is equal to
$\{0,1\}$, $\{0,a,1\}$, $\{0,b,1\}$ or $\M_1$. Since $\{0,a,1\}$ is isomorphic to $\{0,b,1\}$,
the algebra $\M'$ is a subdirect product of algebras isomorphic to $\{0,1\}$, $\{0,a,1\}$
and $\M_1$. It remains to check the conditions
from the definition of the Boolean product. First we will check the conditions concerning the equalizers.
Let $\varphi,\psi\in M$ and let $K=\{y\in Y\mid \varphi(y)=\psi(y)\}$.
We claim that $K\cup f(K)=\{x\in X\mid \varphi(x)=\psi(x)\}$.
Indeed, if $x=f(y)$ for some $y\in K$, then
$\varphi(x)=\varphi(f(y))=f(\varphi(y))=f(\psi(y))=\psi(f(y))=\psi(x)$.
For the reverse set containment, if $\varphi(x)=\psi(x)$, then $x\in Y$ (and hence $x\in
K$) or $f(x)\in Y$ (hence $f(x)\in K$ and $x\in f(K)$). The set
$\{x\in X\mid \varphi(x)=\psi(x)\}$ is clopen in $\X$, because it is
a union of sets $\{x\in X\mid \varphi(x)=c\}\cap\{x\in X\mid
\psi(x)=c\}$ for $c\in\{0,1,a,b\}$. (And these sets are clopen
because $\varphi$ and $\psi$ are continuous.) By the proof of Lemma \ref{topol},
the set $K$ is clopen in $Y$.

Now let $K\subseteq Y$ be clopen in $Y$ and $\varphi,\psi\in M$.
Then the set $K\cup f(K)$ is clopen in $\X$. We define $\rho: X\to
M_1$ as follows:
$$\rho(x)=\left\{\begin{array}{ll} \varphi(x)\quad{\rm if}\ x\in K\cup f(K),\\
                               \psi(x)\quad\text{if}\ x\notin K\cup f(K).
\end{array}\right.$$
Since both $K\cup f(K)$ and its complement are clopen, closed under
$f$, and any pair of elements $x\in K\cup f(K)$ and $y\notin K\cup f(K)$ is
incomparable with respect to $\preccurlyeq$, the mapping $\rho$ is a morphism, hence
$\rho\in M$. Clearly, $\rho$ coincides with $\varphi$ on $K$ and
with $\psi$ on $Y\setminus K$.
\end{proof}

For finite $Y$, the Boolean products are the usual direct products, so
we have the following result as an immediate consequence:

\begin{coro} A finite de Morgan algebra $\M$ is a perfect extension of $B(\M)$ if and only if it is
a direct product of finitely many copies of $\{0,1\}$, $\{0,a,1\}$
and $\M_1$.
\end{coro}

\paragraph{Acknowledgements.}
The first author acknowledges his appointment as a Visiting Professor at the University of Johannesburg in the years 2020-2023. Both authors express their thanks to the editor and the referee for many useful suggestions concerning the final improvement of the text.

%%%%%%%%%%%%%%%%%%%%%%%%%%%%%%%%%%%%
\end{document}